\documentclass[12pt]{amsart}
\usepackage{amssymb,amsmath,amsfonts,epsfig,latexsym, xypic, hyperref}

\usepackage{enumerate}
\usepackage{tikz}

\tikzset{my_dot/.style={fill, circle, inner sep=0pt,minimum size=3pt}}
\tikzset{my_node/.style={fill, circle, inner sep=0pt,minimum size=3pt}}
\tikzset{inv/.style={fill, circle, inner sep=0pt,minimum size=0pt}}

\newtheorem{theorem}{Theorem}[section]
\newtheorem{definition}[theorem]{Definition}
\newtheorem{proposition}[theorem]{Proposition}

\newtheorem*{theorem*}{Theorem}
\newtheorem*{corollary*}{Corollary}

\numberwithin{equation}{subsection}

\theoremstyle{definition}
\newtheorem{remark}[theorem]{Remark}
\newtheorem{example}[theorem]{Example}

\newcommand{\iso}{\cong}

\def\Q{\ensuremath{\mathbb{Q}}}
\def\R{\ensuremath{\mathbb{R}}}
\def\Z{\ensuremath{\mathbb{Z}}}

\def\Sets{\ensuremath{\mathsf{Sets}}}

\def\uX{\ensuremath{\underline{X}}}

\DeclareMathOperator{\Aut}{Aut}

\DeclareMathOperator{\im}{im}

\DeclareMathOperator{\op}{op}

\def\trop{\mathrm{trop}}

\newcommand{\double}{\genfrac..{0pt}1
{\raise -2pt\hbox{$\scriptstyle\longrightarrow$}}{\raise 4pt\hbox
{$\scriptstyle\longrightarrow$}}}

\begin{document}

\title{Tropical moduli spaces as symmetric $\Delta$-complexes}

\author{Daniel Allcock}
\address{UT Department of Mathematics, 2515 Speedway, RLM 8.100, Austin, TX 78712}
\email{allcock@math.utexas.edu}

\author{Daniel Corey}
\address{UW Department of Mathematics, Van Vleck Hall, 480 Lincoln Dr., Madison, WI  53706}
\email{dcorey@math.wisc.edu}

\author{Sam Payne}
\address{UT Department of Mathematics, 2515 Speedway, RLM 8.100, Austin, TX 78712}
\email{sampayne@utexas.edu}

\begin{abstract}
We develop techniques for studying fundamental groups and integral singular homology of symmetric $\Delta$-complexes, and apply these techniques to study moduli spaces of stable tropical curves of unit volume, with and without marked points. As one application, we show that $\Delta_g$ and $\Delta_{g,n}$ are simply connected, for $g \geq 1$.  We also show that $\Delta_3$ is homotopy equivalent to the $5$-sphere, and that $\Delta_4$ has $3$-torsion in $H_5$.
 \end{abstract}

\maketitle

\section{Introduction}

In this paper, we study the topology of $\Delta_g$, the link of the vertex in the moduli space $M_g^\trop$ of stable tropical curves of genus $g \geq 2$. The full moduli space $M_g^\trop$, as defined and studied in \cite{BrannettiMeloViviani11, acp}, is contractible because it is a cone, but the link $\Delta_g$ has exceedingly rich and interesting topology.  Indeed, the rational homology of $\Delta_g$ is naturally identified with both Kontsevich's graph homology and also with the top weight cohomology of the algebraic moduli space $M_g$; applications of this perspective include the proof that $\dim_\Q H^{4g-6}(M_g; \Q)$ grows exponentially with $g$ \cite{CGP1}. The homotopy type of $\Delta_g$ is an invariant of the algebraic moduli stack $\mathcal{M}_g$ \cite{Harper17}, and here we focus on its integral homology and homotopy groups.  

The points of $\Delta_g$ naturally correspond with isomorphism classes of stable tropical curves of genus $g$ and volume 1.  Here a tropical curve is a tuple $\mathbf{G} = (G, \ell, w)$, where $G$ is a connected graph (possibly with loops and multiple edges), $\ell$ is a positive real length assigned to each edge, and $w$ is a non-negative integer weight assigned to each vertex.  The genus is $g(\mathbf{G}) = h^1(G) + \sum_{v \in V(G)} w(v)$, the volume is $\mathrm{Vol}(\mathbf{G}) = \sum_{e \in E(G)} \ell(e)$, and the stability condition says that every vertex $v$ of weight zero has valence at least 3.  

The space $\Delta_g$ has the structure of a symmetric $\Delta$-complex, i.e., it is constructed by gluing quotients of simplices by finite groups in a manner analogous to the construction of a $\Delta$-complexes by gluing simplices \cite[Section~4.2]{CGP1}.  We recall the precise definition of a symmetric $\Delta$-complex in Section~\ref{sec:symmcomplex}.  In Sections~\ref{sec:fundamentalgroup}-\ref{sec:applications}, we develop tools for computing fundamental groups of symmetric $\Delta$-complexes and apply them to $\Delta_g$.

\begin{theorem} \label{thm:simplyconnected}
The space $\Delta_g$ is simply connected.
\end{theorem}

In addition, we develop tools for computing homology of symmetric $\Delta$-complexes with integer coefficients and for proving contractibility of subcomplexes.  We apply these to $\Delta_g$, with particular attention to examples where $g$ is small.  The space $\Delta_2$ has two maximal cells, each homeomorphic to a triangle, and these are glued along an edge.  In particular, $\Delta_2$ is contractible. Using the comparison to graph homology, it is also easy to see that $\Delta_3$ and $\Delta_4$ have the rational homology of $S^5$ and a point, respectively.

\begin{theorem} \label{thm:sphere}
The space $\Delta_3$ is homotopy equivalent to $S^5$.
\end{theorem}

\begin{theorem} \label{thm:torsion}
The reduced integral homology $\widetilde H_k(\Delta_4;\Z)$ contains nontrivial $3$-torsion for $k = 5$ and nontrivial $2$-torsion for $k = 6, 7$. It vanishes for $k \neq 5, 6, 7$.
\end{theorem}

\noindent In particular, $\Delta_4$ is not contractible. 

\smallskip

Our methods apply equally well to moduli spaces of stable tropical curves with marked points.  For non-negative integers $g$ and $n$, such that $2g-2+n > 0$, the moduli space $M_{g,n}^\trop$ of stable tropical curves of genus $g$ with $n$ marked points has pure dimension $3g-3+n$. In particular, when $3g-3+n > 0$, the link $\Delta_{g,n}$ is not empty.   A point of $\Delta_{g,n}$ corresponds to a tropical curve $(G,\ell, w)$ of volume 1, together with a marking function $m \colon \{1,\ldots, n\} \to V(G)$, and the stability condition says that, for every vertex $v$ of weight zero, the valence of $v$ plus $|m^{-1}(v)|$ is at least 3. The topology of $\Delta_{g,n}$ is relatively simple when $g$ is small.  Indeed, $\Delta_{0,n}$ is homotopy equivalent to a wedge sum of $(n-2)!$ spheres of dimension $n - 4$ \cite{Vogtmann90}.  Also, $\Delta_{1,n}$ is contractible for $n = 1, 2$, and homotopy equivalent to a wedge sum of $(n-1)!/2$ spheres of dimension $n-1$, for $n \geq 3$ \cite[Theorem~1.2]{cgp3}.   Note, in particular, that $\Delta_{0,4}$ is not connected, and $\Delta_{0,5}$ is connected but not simply connected.

\begin{theorem} \label{thm:simplyconnected-markedpoints}
The space $\Delta_{g,n}$ is simply connected for $(g,n) \neq (0,4), (0,5)$.
\end{theorem}

The $p$-skeleton of a symmetric $\Delta$-complex is the union of cells of dimension at most $p$. In Section~\ref{sec:fundamentalgroup} we show that the fundamental group of a symmetric $\Delta$-complex is generated by loops in the $1$-skeleton.  However, for $k > 1$, we give examples showing that $\pi_k$ (resp.\ $H_k$) is not generated by spheres (resp.\ cycles) in the $k$-skeleton. Even for $k = 1$, the relations among spheres (resp.\ cycles) in the $k$-skeleton are not necessarily generated by boundaries of balls (resp.\ chains) in the $(k+1)$-skeleton.  In particular, the relations in $\pi_1$ among the classes of loops in the $1$-skeleton are not always generated by boundaries of discs in the $2$-skeleton.

In Section~\ref{sec:spectral} we describe the spectral sequence associated to the filtration by $p$-skeleta.  Taken with rational coefficients, this spectral sequence recovers the rational cellular homology theory developed in \cite{CGP1}. With integer coefficients, it provides new information on torsion in homology of symmetric $\Delta$-complexes.  We refine this approach by considering an analogous spectral sequence relative to a subcomplex, which is essential for the proof of Theorem~\ref{thm:torsion}. The idea of computing homology of moduli spaces of tropical curves relative to large contractible subcomplexes is not new; contractibility of $\Delta_{g,n}^{br}$, the closure of the locus of tropical curves with bridges, is essential to the main results in \cite{cgp3}.  Here we use an even larger contractible subcomplex, given by Theorem~\ref{thm:contractible}.

\smallskip

The techniques we develop to study $\Delta_{g,n}$ as a symmetric $\Delta$-complex apply to more general filtered spaces, in which a $p$-skeleton is obtained from a $(p-1)$-skeleton as the mapping cone for a continuous map from a disjoint union of quotients of $(p-1)$-spheres by finite subgroups of the orthogonal group.  We call these \emph{symmetric CW-complexes}, and Theorems~\ref{thm:fundamentalgroup}, \ref{thm:cellularhom}, and \ref{thm:cellularhom2} are proved for this more general class of spaces.  Symmetric  CW-complexes that are not symmetric $\Delta$-complexes appear naturally in algebraic geometry as dual complexes of boundary divisors in non-simplicial toroidal compactifications, such as the Voronoi compactifications of moduli spaces of abelian varieties.

\begin{remark}
The space $\Delta_g$ has several natural interpretations in algebraic geometry, low-dimensional topology, and geometric group theory; see \cite[pp.~1--2]{CGP1}. Notably, it is the quotient of Harvey's complex of curves on a closed orientable surface of genus $g$ by the action of the mapping class group, and is also the quotient of the simplicial completion of Culler-Vogtmann Outer Space by the action of $\mathrm{Out}(F_g)$. It is homotopy equivalent to the one point completion of the quotient of Outer Space itself by the action of $\mathrm{Out}(F_g)$.  Nevertheless, the rational homology of $\Delta_g$ has little relation to the rational homology of $\mathrm{Out}(F_g)$ \cite[Section~7]{cgp3}, and we do not know of any greater relation between the torsion homologies.  For an alternate proof of Theorem~\ref{thm:simplyconnected}, using ideas from geometric group theory, see Remark~\ref{rem:alternateproof}.
\end{remark}

\begin{remark}
Theorem~\ref{thm:torsion} is not the first observed instance of torsion in the homology of the spaces $\Delta_{g,n}$.  Indeed, $H_*(\Delta_{2,n},\Z)$ has torsion in degree $n+1$, for odd $n \geq 5$ \cite{Chan15}.  However, it is the first instance of torsion in degree less than $\max\{2g-1,2g-3+n\}$, the bound on low degree rational homology of $\Delta_{g,n}$ induced by vanishing of high degree cohomology on the algebraic moduli space $\mathcal{M}_g$ \cite[Theorem~1.6]{cgp3}.
\end{remark}

\noindent \textbf{Acknowledgments.} We thank Ben Blum-Smith for helpful conversations related to quotients of spheres by subgroups of permutation groups and A.~Putman for explaining the alternate proof that $\Delta_g$ is simply connected, presented in Remark~\ref{rem:alternateproof}.  The work of DA is supported in part by Simons Foundation Collaboration Grant 429818, the work of DC is supported in part by NSF RTG Award DMS--1502553, and the work of SP is supported in part by NSF DMS--1702428.

\section{Structure of symmetric $\Delta$-complexes} \label{sec:symmcomplex}

We begin by recalling the definition of symmetric $\Delta$-complexes, following \cite[Section~4]{CGP1}.  Let $I$ be the category with objects $[p] = \{ 0, \ldots, p \}$, for $p \geq 0$, with morphisms given by injective maps.  Recall that a symmetric $\Delta$-complex is a functor $X \colon I^{\op} \to \Sets$.  Such a functor is determined by a set $X_p = X([p])$ for each $p \geq 0$, actions of the symmetric group $\mathfrak{S}_{p+1}$ on $X_p$ for all $p$, and face maps $d_i \colon X_p \to X_{p-1}$ for $p \geq 1$, obtained by applying the functor $X$ to the unique order-preserving injective map $[p-1] \to [p]$ whose image does not contain $i$.  The face maps satisfy the usual simplicial identities as well as a compatibility with the symmetric group action.

An injection $\theta \colon [p] \to [q]$ determines an inclusion of standard simplices $\theta_* \colon \Delta^{p} \to \Delta^{q}$, whose image is the $p$-face with vertices corresponding to the image of $\theta$.  The \emph{geometric realization} of a symmetric $\Delta$-complex $X$ is 
\begin{equation}\label{eq:realization}
  |X| = \Big(\coprod_{p = 0}^\infty X_p \times \Delta^p\Big) \big / \sim,
\end{equation}
where $\sim$ is the equivalence relation generated by $(x,\theta_* a) \sim (\theta^* x, a)$. 
Each $x \in X_p$ determines a map of topological spaces $x\colon \Delta^p \to |X|$, which factors through the quotient of $\Delta^p$ by the stabilizer $H_x < \mathfrak{S}_{p+1}$, and $X$ may be recovered from the topological space $|X|$ together with this set of maps from simplices.

\medskip

Let $\uX_p \subset X_p$ be a subset consisting of one representative of each $\mathfrak{S}_{p+1}$-orbit.  Then $|X|$ is partitioned into \emph{cells}
\[
|X| = \coprod_{p} \coprod_{x \in \uX_p} (\Delta^p)^\circ / H_x,
\]
each isomorphic to the quotient of an open simplex by a linear finite group action.  Note that the closure of each $p$-cell meets only finitely many cells, each of dimension less than $p$.  The properties of this stratification are hence closely analogous to the properties of a CW-complex, except that the cells are quotients of open balls by finite groups, rather than ordinary open balls.  We capture this analogy with the following definition.

\begin{definition} \label{def:gcw}
A \emph{symmetric CW-complex} is a Hausdorff topological space $X$ together with a partition into locally closed \emph{cells}, such that, for each cell $C \subset X$, there is a continuous map from the quotient of the closed unit ball in some $\R^p$ by a finite subgroup of the orthogonal group such that
\begin{enumerate}
\item The quotient of the open unit ball maps homeomorphically onto $C$, and
\item The image meets only finitely many cells, each of dimension less than $p$.
\end{enumerate}
We require furthermore that a subset of $X$ is closed if and only if its intersection with closure of each cell is closed.  
\end{definition}

We say that a symmetric CW-complex is \emph{finite} if it has only finitely many cells.  All of the applications we consider involve only finite symmetric CW-complexes.

\begin{definition}
The $p$-skeleton of a symmetric CW-complex, denoted $X^{(p)} \subset X$, is the union of its cells of dimension at most $p$.  
\end{definition}

Suppose $X$ has $p$-cells $\{ C_i : i \in I \}$.  Let $G_i \leq O(p)$ be a subgroup of the orthogonal group for which there is a continuous map $B^p / G_i \to X$ taking $(B^p)^\circ / G_i$  homeomorphically onto $C_i$.  Note that $B^p / G_i$ is cone shaped around the image of the origin, and that $X^{(p)}$ is naturally identified with the mapping cone of the attaching map
\[
\coprod_{i \in I} \partial B^p / G_i \to X^{(p-1)}.
\]
The total space $X$ is the direct limit of its skeleta $X^{(p)}$.

\begin{example}
Let $X$ be a symmetric $\Delta$-complex.  Then $|X|$, together with its partition into cells, is a symmetric CW-complex.
\end{example}

\begin{example} \label{ex:generalizedcone}
Let $\Delta$ be a generalized cone complex, as defined in \cite{acp}, obtained as the colimit of a diagram of rational polyhedral cones with face maps.  Then the geometric realization $|\Delta|$ is partitioned into cells by the images of the relative interiors of the cones in the diagram.  The link of the vertex, with its induced partition into cells, is a symmetric CW-complex.
\end{example}

\begin{remark}
Example~\ref{ex:generalizedcone} shows that symmetric CW-complexes appear naturally in algebraic geometry, as dual complexes of boundary divisors in toroidal embeddings.  In cases where the toroidal embedding is non-simplicial, as is the case in such well-studied examples as the perfect cone and second Voronoi toroidal compactifications of moduli spaces of abelian varieties, the resulting symmetric CW-complex is not a symmetric $\Delta$-complex.  While the applications to moduli spaces of tropical curves stated in the introduction require only the special case of symmetric $\Delta$-complexes, we expect that symmetric CW-complexes will be useful for future applications, e.g., for studying the top weight cohomology of the moduli space of abelian varieties.  Since the proofs of our basic technical results (Theorems~\ref{thm:fundamentalgroup}, \ref{thm:cellularhom}, and \ref{thm:cellularhom2}) work equally well for symmetric CW-complexes, we present them in this greater level of generality.
\end{remark}

\section{How attaching $p$-cells affects homology and homotopy groups of symmetric CW-complexes} \label{sec:fundamentalgroup}

Let $X$ be a symmetric CW-complex obtained from a subcomplex $X'$ by attaching cells of dimension $p$.  In stark contrast with ordinary CW-complexes, the pair $(X, X')$ is not necessarily $(p-1)$-connected, i.e., the inclusion $X' \hookrightarrow X$ need not induce isomorphisms on homotopy groups $\pi_i$ for $i < p-1$ and a surjection for $i = p-1$.  See Examples~\ref{ex:RPn}-\ref{ex:notsurjective}.  Hence, the classical cellular approximation theorems for CW-complexes do not extend to symmetric CW-complexes.  The following theorem is the key technical step in the proof of Theorems~\ref{thm:simplyconnected} and \ref{thm:simplyconnected-markedpoints}; it shows that cellular approximation holds in dimension 1, even though it fails in all higher dimensions.

\begin{theorem} \label{thm:fundamentalgroup}
Let $X$ be a finite symmetric CW-complex with $x \in |X^{(1)}|$.  Then the natural map $\pi_1(|X^{(1)}|,x) \to \pi_1(|X|,x)$ is surjective.
\end{theorem} 

\begin{proof}
It suffices to show that, if $X$ is obtained from a subcomplex $X'$ by attaching a single cell of dimension $n \geq 2$, then the induced map $\pi_1(|X'|,x) \to \pi_1(|X|,x)$ is surjective.  We may assume $X$ and $X'$ are connected.  In this case, $|X|$ is the identification space obtained from $|X'|$ and $B^n/G$ by identifying the points of $S^{n-1}/G$ with their images in $|X'|$ under the attaching map, where $G \leq O(n)$ is finite.  We write $0$ for the origin in~$B^n$ and express $|X|$ as the union of the open sets $|X|-\{0\}$ and $(B^n-S^{n-1})/G$. Their intersection is homeomorphic to $\R\times(S^{n-1}/G)$, hence connected.  By van Kampen's theorem, $\pi_1(|X|)$ is generated by the images under inclusion of $\pi_1(|X|-\{0\})$ and $\pi_1((B^n-S^{n-1})/G)$. The first of these coincides with the image under inclusion of $\pi_1(|X'|)$. This is because the radial deformation retraction of $B^n-\{0\}$ onto $S^{n-1}$ is $G$-equivariant and therefore induces a deformation retraction from $(B^n-\{0\})/G$ into $S^{n-1}/G$.  Since $\pi_1((B^n-S^{n-1})/G)$ is trivial, $\pi_1(|X|)$ is generated by the image under the inclusion of $\pi_1(|X'|)$, as required.
\end{proof}

It follows from Theorem~\ref{thm:fundamentalgroup} that $H_1$ is also generated by loops in the $1$-skeleton of a finite-dimensional symmetric CW-complex. This surjectivity on $H_1$ and $\pi_1$ looks like a weak result, but it is the only general result of this nature that one can expect. Suppose $k>0$. We will exhibit a finite symmetric $\Delta$-complex whose dimension $n$ is larger than $k+1$, such that the inclusion of $|X^{(n-1)}|$ into $|X|$ induces non-injective maps on $H_k$ and $\pi_k$ and non-surjective maps on $H_{k+1}$ and $\pi_{k+1}$. We begin with symmetric CW-complex examples because they make the essentials more visible. See Remark~\ref{rem:higherk} for an example in nature of a symmetric $\Delta$-complex~$X$ such that $k$th homology and homotopy groups of the  $k$-skeleton do not surject onto those of~$X$.

\begin{example}[Attaching a symmetric $n$-cell can kill $\pi_1$ and $H_k$ for odd $k<n-1$]
    \label{ex:RPn}
    Suppose $n>2$, let $G\leq O(n)$ be generated by the negation map, and consider $B^n/G$. The  CW complex structure on $|X'|=S^{n-1}/G\iso\R P^{n-1}$ is not important.  We regard $B^n/G$ as a single symmetric $n$-cell, attached to $|X'|$ in the obvious way.  The result is contractible.  Therefore, attaching this symmetric $n$-cell kills not just $H_{n-1}(\R P^{n-1};\Z)$ but also
    $\pi_1(\R P^{n-1})\iso\Z/2\Z$ and $H_{i}(\R P^{n-1};\Z)\iso\Z/2\Z$ where $i< n-1$ is positive and odd. %So attaching a symmetric $n$-cell can kill $\pi_1$ and $H_{\mathrm{odd}<n-1}$. 
\end{example}

\begin{example}[Attaching a symmetric $n$-cell can kill $\pi_k$ and $H_k$ for $2\leq k<n-1$]
    \label{ex:RPnSuspended}
    As a variation on the previous example, take $2\leq k\leq n-2$ and let $g\in O(n)$ act on $\R^n$ by fixing pointwise a $(k-1)$-dimensional subspace and acting by negation on its orthogonal complement. Write $G\iso\Z/2\Z$ for the group generated by~$g$.  Again $B^n/G$ is contractible, but this time $|X'|=S^{n-1}/G$ is the join of $\R P^{n-k}$ and $S^{k-2}$.  Taking the join of a space with $S^{k-2}$ is the same as (unreducedly) suspending it $k-1$ times. So the reduced homology of $|X'|$ is obtained by shifting the reduced homology of $\R P^{n-k}$ up $k-1$ degrees. In particular, the first nonzero homology group of $|X'|$ in positive degree is $H_k(|X'|;\Z)\iso H_1(\R P^{n-k};\Z)\iso\Z/2\Z$. Since we chose $k\geq2$, we always suspend at least once, so $|X'|$ is simply connected. Combined with the Hurewicz theorem, this shows that the first nontrivial homotopy group of $|X'|$ is $\pi_{k}(|X'|)\iso H_{k}(|X'|;\Z)\iso\Z/2\Z$. So the maps on $H_k$ and $\pi_k$ induced by
    $|X'|\to|X|$ are not injective, even though $X$ is obtained from $X'$ by attaching a single symmetric $n$-cell.
\end{example}

The next few examples are quotients of a standard simplex. Let $G \leq \mathfrak{S}_{n+1}$. Then the quotient $X = \Delta^n / G$ naturally inherits the structure of a symmetric $\Delta$-complex, with $X_q$ being the set of $G$-orbits of $q$-faces of $\Delta^n$.  Then $|X^{(n-1)}|$ is the quotient of the boundary of the simplex by~$G$. The geometric realization $|X|$ is contractible because it is the cone over $|X^{(n-1)}|$.

\begin{example}[Analogue of Example~\ref{ex:RPnSuspended} for symmetric $\Delta$-complexes]
    \label{ex:RPnPolyhedral} 
    Suppose $k\geq3$ is given, choose $\ell$ satisfying $3\leq  \ell \leq k$, and set $n=k + \ell -1$. Consider the subgroup $G$ of $O(n)$ generated by the involution $g$ that fixes all but the first $2 \ell$ vertices of $\Delta^n$, which it permutes by $(01)(23)\cdots(2\ell-2\,\, 2\ell-1)$. This makes sense since $2 \ell \leq n+1$. We make the affine span of $\Delta^n\subseteq\R^{n+1}$ into a  vector space~$V$ by taking its barycenter as the origin. It is easy to see that $g$ acts on $V$ by negating an $\ell$-dimensional subspace and pointwise fixing its orthogonal complement, which has dimension $k-1$. We can identify the unit sphere in $V$ with the boundary of $\Delta^n$ by radial projection.  This extends to a $G$-equivariant identification of  $\Delta^n$ with the unit ball in~$V$.  Therefore every feature of the previous example carries over to the symmetric $\Delta$-complex $\Delta^n/G$. In particular, the maps on $\pi_k$ and $H_k$ induced by $|(\Delta^n/G)^{(n-1)}|\to|\Delta^n/G|$ are not injective. We needed the restriction $\ell \geq 3$ to get $n>k+1$, and we needed $\ell \leq k$ so that the permutation of the vertices of $\Delta^n$ makes sense.  Hence these examples only occur for $k\geq3$. For $k=1,2$ we refer to the following examples.
\end{example}

\begin{example}[Attaching a symmetric $3$-simplex can kill $\pi_1$ and $H_1$]
    \label{ex:FourCycle}
    We construct a $3$-dimensional contractible symmetric $\Delta$-complex whose $2$-skeleton is not simply connected. Take $G\iso\Z/4\Z$ to be generated by an element of $O(3)$ that permutes the vertices of $\Delta^3$ cyclically.  The quotient of the boundary can be worked out by taking one facet as a fundamental domain for~$G$ and working out the edge pairings. It turns out to be $\R P^2$.  Taking $X=\Delta^3/G$, it follows that $|X|$ is simply connected, even though its $2$-skeleton is not:  $\pi_1(|X^{(2)}|)\iso\pi_1(\R P^2)\iso\Z/2\Z$. 
\end{example}

\begin{example}[Attaching a symmetric $4$-simplex can kill $\pi_2$ and $H_2$] 
    \label{ex:FourCycleFourSimplex}
    We construct a $4$-dimensional contractible symmetric $\Delta$-complex whose $3$-skeleton is simply connected and has nontrivial $H_2$ (hence $\pi_2$). Take $X=\Delta^4/G$ where $G\iso\Z/4\Z$ fixes one vertex and permutes the others cyclically. This example contains the previous one; the quotient of the boundary of $\Delta^4$ is the unreduced suspension of $\R P^2$.  The suspension points are the fixed points of~$G$, namely the fixed vertex and the barycenter of the opposite facet.  
\end{example}

\begin{example}[Non-surjectivity of $\pi_k$ and $H_k$]
    \label{ex:notsurjective}
    Take $X$ as in any of the previous examples, and define $Z$ as the symmetric CW-complex or symmetric $\Delta$-complex obtained by identifying the two copies of $X$ along their  $(n-1)$-skeleta. So $Z^{(n-1)}=X'$ in the notation of the previous examples.  Let $k>1$ be the smallest degree for which $H_k(|X'|;\Z)\neq0$.  Since $Z$ is the (unreduced) suspension of~$X'$, we have $H_{k+1}(|Z|;\Z)\neq0$.  The Hurewicz theorem shows that $\pi_{k+1}(|Z|)$ is also nonzero.  On the other hand, the natural maps $H_{k+1}(|X'|;\Z)\to H_{k+1}(|Z|;\Z)$ and $\pi_{k+1}(|X'|)\to\pi_{k+1}(|Z|)$ are the zero maps because $|X'|\to|Z|$ factors through the contractible space $|X|$.  So the induced maps on $H_{k+1}$ and $\pi_{k+1}$ are not surjective.
\end{example}

\section{Computing homology using the filtration by skeleta} \label{sec:spectral}

Let $A$ be an abelian group.  Any finite filtration of a topological space $Y$ by subspaces $$\emptyset = Y^{-1} \subset Y^0 \subset Y^1 \subset \cdots \subset Y^n = Y$$ induces a filtration on the singular chain complex with coefficients in $A$,
\[
0 \subset C(Y^0; A) \subset C(Y^1; A) \subset \cdots \subset C(Y^n; A) = C(Y; A).
\]
This filtration on $C(Y)$ gives rise to a spectral sequence with 
\[
E_0^{p, q} = C_{p+q}(Y^p, Y^{p-1}; A) = C_{p+q}(Y^p; A)/ C_{p+q}(Y^{p-1}; A),
\]
and
\[
E_1^{p, q} = H_{p+q}(Y^p, Y^{p-1}; A),
\]
that converges to
\[
E_\infty^{p,q} = \frac{\im (H_{p+q}(Y^p; A) \to H_{p+q} (Y; A))}{ \im (H_{p+q}(Y^{p-1}; A) \to H_{p+q}(Y; A))}.
\]

\begin{remark}
The higher differentials in this spectral sequence may be understood as follows.  An element of $E^1_{p,q}$ is represented by a $(p+q)$-chain $\sigma$ in $Y^p$ whose boundary $\partial \sigma$ is a $(p+q-1)$-cycle in $Y^{p-1}$.  If $\partial \sigma$ is contained in $Y^{p-r}$ but not in $Y^{p-r-1}$, then $[\sigma]$ survives to $E_r$, and $d_r \colon E_r^{p,q} \to E_r^{p-r, q+r-1}$ maps $[\sigma]$ to $[\partial \sigma]$ in the surviving subquotient of $H_{p+q-1}(Y^{p-r}, Y^{p-r-1}; A)$.
\end{remark}

This spectral sequence will be our main tool for understanding the homology of symmetric CW-complexes.  When applied to the filtration by skeleta, it gives the following.

\begin{theorem} \label{thm:cellularhom}
Let $X$ be a finite-dimensional symmetric CW-complex whose $p$-cells are $\{ C_{i} \cong (B^p)^\circ / G_i : i \in I^p \}$, for some finite subgroups $G_i \subset O(p)$.  Then there is a spectral sequence with 
\[
E_1^{0,q} = H_q (|X^{(0)}|; A),  \mbox{ \ \ and \ \ } E_1^{p,q} = \bigoplus_{i \in I^p} \widetilde{H}_{p+q-1}(S^{p-1} / G_i; A),
\]
 for $p \geq 1$, that converges to
\[
E_\infty^{p,q} = \frac{\im (H_{p+q}(|X^{(p)}|; A) \to H_{p+q} (|X|; A))}{ \im (H_{p+q}(|X^{(p-1)}|; A) \to H_{p+q}(|X|; A))}.
\]
\end{theorem}

\begin{proof}
Consider the spectral sequence associated to the filtration of $X$ by its skeleta $X^{(p)}$, which has $E_1^{p, q} = H_{p+q}(|X^{(p)}|, |X^{(p-1)}|; A)$.  For $p > 0$, each connected component of $|X^{(p)}| / |X^{(p-1)|}$ is homeomorphic to the wedge sum of unreduced suspensions $S (S^{(p-1)}/G_i)$ associated to $p$-cells $C_i = B^p / G_i$ in that component, and the theorem follows.
\end{proof}

\noindent If we take $A = \Q$, then $E_1^{p,q}$ vanishes for $q \neq 0$, so we get a single chain complex, the $q = 0$ row of $E_1$, that computes the rational homology of $|X|$.  When $X$ is a symmetric $\Delta$-complex, this is precisely the rational cellular chain complex presented in \cite{CGP1}.

For our applications to $\Delta_{g,n}$, which typically has many cells and contains a large contractible subcomplex, we use the following variant of Theorem~\ref{thm:cellularhom}.

\begin{theorem} \label{thm:cellularhom2}
Let $X$ be a finite-dimensional symmetric CW-complex, let $Z \subset X$ be a subcomplex, and let $\{ C_{i} \cong (B^p)^\circ / G_i : i \in I^p \}$, be the cells of $X$ that are not in $Z$.  Then there is a spectral sequence with 
\[
E_1^{0,q} = H_q (|Z| \cup X^{(0)}; A), \mbox{ \ \ and \ \ } E_1^{p,q} = \bigoplus_{i \in I^p} \widetilde{H}_{p+q-1}(S^{p-1} / G_i; A),
\]
for $p \geq 1$, that converges to
\[
E_\infty^{p,q} = \frac{\im (H_{p+q}(|X^{(p)}|; A) \to H_{p+q} (|X|; A))}{ \im (H_{p+q}(|X^{(p-1)}|; A) \to H_{p+q}(|X|; A))}.
\]
\end{theorem}

\begin{proof}
The argument is identical to the proof of Theorem~\ref{thm:cellularhom}, using the filtration of $X$ by the subcomplexes $W^{(p)} = Z \cup X^{(p)}$.  
\end{proof}

\section{Quotients with reflections}

As explained above, a symmetric CW-complex is filtered by its $p$-skeleta, and the associated graded of the induced filtration on singular homology is the abutment of a spectral sequence whose $E_1$-page is a direct sum of homology groups of quotients of spheres by finite subgroups of orthogonal groups.  The topology of such quotients of spheres is understood in only a few special cases.  Swartz considered the case where the subgroup of the orthogonal group is isomorphic to $(\Z/2\Z)^k$ \cite{Swartz02}, and Lange classified the cases where the quotient is a PL-sphere \cite{Lange16} or a topological sphere \cite{Lange19}. See also \cite{BlumSmithMarques18} for applications of Lange's results to quotients of $S^{n-1}$ by subgroups of the permutation group $\mathfrak{S}_{n+1}$.

\begin{proposition} \label{prop:reflection}
Let $H < O(n)$ be a finite subgroup that contains a reflection.  Then $S^{n-1} / H$ is contractible.
\end{proposition}

\begin{proof}
Let $K < H$ be the subgroup generated by reflections.  
    Each reflection in~$K$ fixes a hyperplane pointwise, and we let $C$
    be a chamber, i.e., the closure of a component of the complement
    of the union of these hyperplanes. 
    Every point in $\R^n$ is $K$-equivalent to exactly one point in $C$ \cite[V.3.3]{Bourbaki68}, and hence $C \to \R^n / K$ is a homeomorphism.  Furthermore, since $K$ is normal in $H$, we have homeomorphisms
\[
    S^{n-1} / H \cong (S^{n-1} / K) / (H/K) \cong (S^{n-1}\cap C) / H_C,
\]
where $H_C < H$ is the stabilizer of $C$.  
    We now show that $(S^{n-1} \cap C) / H_C$ is contractible.  

 There is a point $v$ in the interior of $C$ that is invariant
    under $H_C$, e.g., choose any point in the interior of $C$ and take the
    sum of the points in its $H_C$-orbit.  Composing straight line flow
    toward $v$ with radial projection to the unit sphere gives a deformation
    retraction of $S^{n-1} \cap C$ to the point $v/|v|$. This deformation
    retraction is $H_C$-equivariant, and hence induces a deformation
    retraction of   $(S^{n-1}\cap C)/H_C$ to a point.
\end{proof}

\section{Applications to moduli spaces of tropical curves} \label{sec:applications}

One strategy for studying the topology of $\Delta_{g,n}$ is to identify large contractible subcomplexes.  The largest contractible subcomplex identified in \cite[Theorem~1.1]{cgp3} is the locus $\Delta_{g,n}^{br}$ of tropical curves with bridges, cut vertices, loop edges, repeated markings, or vertices of positive weight; it is the closure of the locus of tropical curves with bridges.  We now use Proposition~\ref{prop:reflection} to produce an even larger contractible subcomplex.

\begin{theorem} \label{thm:contractible}
The subcomplex $\Delta^{bm}_{g,n} \subset \Delta_{g,n}$ parametrizing tropical curves with bridges, cut vertices, loops, repeated markings, vertices of positive weight, or multiple edges is contractible, for $g \geq 1$.
\end{theorem}

\noindent The superscript $bm$ reflects the fact that this subcomplex is the closure of the locus of tropical curves with \emph{bridges} or \emph{multiple} edges.

\begin{proof}
First, note that contracting an edge in a graph with a multiple edge produces a graph with either a loop edge or a multiple edge.  Hence $\Delta^{bm}_{g,n}$ is a subcomplex. 

Note that $\Delta^{bm}_{g,n}$ is obtained from the subcomplex $\Delta^{br}_{g,n}$ as an iterated mapping cone, for a finite sequence of continuous maps from quotients of spheres $S^{p-1} / \Aut(G)$, where $G$ is a graph with multiple edges.  Interchanging a pair of edges between the same endpoints acts by a reflection, and hence the quotient $S^{p-1} /\Aut(G)$ is contractible, by Proposition~\ref{prop:reflection}.

If $Y$ is contractible and $f \colon Y \to Z$ is continuous, then the inclusion of $Z$ into the mapping cone of $f$ is a homotopy equivalence. Applying this observation to the iterated mapping cone construction discussed above, we see that the inclusion of $\Delta^{br}_{g,n}$ in $\Delta^{bm}_{g,n}$ is a homotopy equivalence.  The subcomplex $\Delta^{br}_{g,n}$ is contractible \cite[Theorem~1.1]{CGP1}, and the theorem follows.
\end{proof}

We now show that $\Delta_{g,n}$ is simply connected, for $(g,n) \neq (0,4), (0,5)$.

\begin{proof}[Proof of Theorem~\ref{thm:simplyconnected-markedpoints}]
First, note that $\Delta_{0,n}$ is simply connected for $n \geq 6$, because it is homotopic to a wedge of spheres of dimension $n-4$ \cite{Vogtmann90}.  We therefore assume that $g \geq 1$.  We claim that the $1$-skeleton of $\Delta_{g,n}$ is contained in $\Delta^{bm}_{g,n}$, for $g \geq 1$.  To see this, note that $\Delta_{g,n}^{(1)}$ parametrizes stable tropical curves with $1$ or $2$ edges.  Either one of the edges is a bridge, all of the edges are loops, or there are two edges that together form a loop.  In particular, the underlying graph has either a loop, bridge, or multiple edges. This proves the claim.  Next, recall that the fundamental group is generated by loops in $\Delta_{g,n}^{(1)}$, by Theorem~\ref{thm:fundamentalgroup}.  If $g \geq 1$, then all such loops can be contracted in $\Delta^{bm}_{g,n}$, by Theorem~\ref{thm:contractible}, and hence $\Delta_{g,n}$ is simply connected.
\end{proof}

\begin{remark} \label{rem:alternateproof}
We briefly sketch an alternate proof that $\Delta_g$ is simply connected, suggested by A.~Putman. A theorem of Armstrong \cite{Armstrong65} says that, when a group $G$ acts simplicially on a simply connected simplicial complex $X$, the fundamental group of the quotient space $X/G$ is the quotient of $G$ by the subgroup generated by elements of $g$ that fix some point of $X$.  In particular, since $\Delta_g$ is the quotient of Harvey's complex of curves on a closed orientable surface of genus $g$ by the action of the mapping class group, and since the mapping class group is generated by Dehn twists, each of which fixes points in the curve complex, the quotient $\Delta_g$ is simply connected.  Variations on this argument are possible as well, e.g., using the description of $\Delta_g$ as the quotient of the simplicial completion of Culler-Vogtmann Outer Space by the action of $\mathrm{Out(F_g)}$, or using the action of the pure mapping class group on the complex of curves on a punctured surface to give an alternate proof of  Theorem~\ref{thm:simplyconnected-markedpoints}.
\end{remark}

Next, we show that $\Delta_3$ is homotopy equivalent to $S^5$.

\begin{proof}[Proof of Theorem~\ref{thm:sphere}]
First, note that $\Delta_3$ is homotopy equivalent to $\Delta_3 / \Delta_3^{bm}$, by Theorem~\ref{thm:contractible}.  Enumerating stable graphs of genus $3$ shows that the only one without bridges, cut vertices, loops, multiple edges, or vertices of positive weight is the complete graph $K_4$.  The cell of $\Delta_3$ corresponding to $K_4$ is the quotient of the open $5$-simplex, whose vertices correspond to the $2$-element subsets of a $4$-element set, by the permutation group $\mathfrak{S}_4$.  

It follows  $\Delta_3 / \Delta_3^{bm}$ is homeomorphic to the unreduced
    suspension $S (S^4 / \mathfrak{S}_4)$.  Each transposition in
    $\mathfrak{S}_4$ acts by a double transposition on the vertices of the
    $5$-simplex.
    This is a rotation in the sense of \cite{Lange16}, meaning that
    its fixed-point set has codimension~$2$. 
    Any quotient of $S^{n-1}$ by a finite group generated by such reflections is PL-homeomorphic to $S^{n-1}$ \cite{Lange16}.  In
    particular, $S^4 / \mathfrak{S}_4$ is homeomorphic to $S^4$, and hence
    $\Delta_3$ is homotopy equivalent to $S(S^4) \cong S^5$.
\end{proof}

We conclude by showing that $\Delta_4$ has nontrivial $3$-torsion in $H_5$ and nontrivial $2$-torsion in $H_6$ and $H_7$.  The proof uses Theorems~\ref{thm:cellularhom2} and \ref{thm:contractible}, together with explicit computations of the integral homology of certain quotients of spheres $S^{p-1}$ by subgroups of the permutation group $\mathfrak{S}_{p+1}$.  These computations were carried out with computer assistance.  We considered the $\Delta$-complex structure on $S^{p-1}/ \mathfrak{S}_{p+1}$ induced by barycentric subdivision on the boundary of $\Delta^p$, and used a python script to generate matrices for the simplicial chain complex.  We then used Magma \cite{Magma} to compute the homology. The code may be found at the following link. 

\begin{center}
	\url{https://github.com/dcorey2814/homologyQuotientSpheres.git} 
\end{center}

\begin{proof}[Proof of Theorem~\ref{thm:torsion}]
We compute the $E_1$-page of the spectral sequence given by Theorem~\ref{thm:cellularhom2}, using the contractible subcomplex $Z = \Delta_4^{bm}$.

By enumerating the stable graphs of genus 4, we see that there are precisely 3 cells in $\Delta_4$ that are not contained in $\Delta_4^{bm}$.  These are the edge graph $G$ of a square pyramid, the edge graph $G'$ of a triangular prism, and the complete bipartite graph $K_{3,3}$.  For each graph, we computed the reduced homology of the corresponding sphere quotient. The nonzero reduced homology groups are as follows.
\begin{align*}
\widetilde H_k (S^6/ \Aut(G) ; \Z) & = \left \{ \begin{array}{ll} \Z/ 4\Z, & \mbox{\ for \ } k = 4, \\
										      \Z / 2\Z, & \mbox{\ for \ } k = 5; \end{array} \right. \\
\widetilde H_k (S^7 / \Aut(G') ; \Z) & = \left \{ \begin{array}{ll} \Z/ 2\Z, & \mbox{\ for \ } k = 5, \\
										      \Z / 2\Z, & \mbox{\ for \ } k = 6;  \end{array} \right. \\ 
\widetilde H_k (S^7/  \Aut(K_{3,3}); \Z) & = \left \{ \begin{array}{ll} \Z/ 3\Z, & \mbox{\ for \ } k = 4, \\
										      \Z / 4\Z, & \mbox{\ for \ } k = 5, \\
										      \Z / 2\Z, & \mbox{\ for \ } k = 6.
										      \end{array} \right.
\end{align*}

Consider the spectral sequence given by Theorem~\ref{thm:cellularhom2} for $X = \Delta_4$ and $Z = \Delta_4^{bm}$.  The nonzero terms on the $E_1$-page, aside from $E_1^{0,0} = \Z$, are:
\begin{align*}
E_1^{7,-2} & = \Z/4\Z, \\
E_1^{7,-1} & = \Z/2\Z,
\end{align*}
and 
\begin{align*}
E_1^{8,-3} & = \Z/3\Z, \\
E_1^{8,-2} & = \Z/4\Z \oplus \Z/2\Z, \\
E_1^{8,-1} & = \Z/2\Z \oplus \Z/2\Z.
\end{align*}
There are no differentials between nonzero terms on $E_r$, for $r > 1$, so the sequence degenerates at $E_2$.  The only differentials between nonzero terms on $E_1$ are $E_1^{8,-2} \to E_1^{7,-2}$ and $E_1^{8,-1} \to E_1^{7,-1}$.  Each source is larger than its target, so we conclude that both $E_\infty^{8,-2}$ and $E_\infty^{8,-1}$ contain nontrivial $2$-torsion. It follows that $H_6(\Delta_4;\Z)$ and $H_7(\Delta_4;\Z)$ contain nontrivial $2$-torsion.  Similarly, we see that $E_\infty^{8,-3} = \Z/3\Z$ and conclude that $H_5(\Delta_4;\Z)$ contains nontrivial $3$-torsion.
\end{proof}

\begin{remark} \label{rem:higherk}
In the proof of Theorem~\ref{thm:torsion}, we have seen that $\Delta_4^{(6)}$ is contained in a contractible subcomplex $\Delta_4^{bm}$.  However, $\Delta_4$ has nontrivial $H_5$ and $H_6$.  We conclude that $H_5(\Delta_4;\Z)$ and $H_6(\Delta_4;\Z)$ are not generated by cycles in the $5$-skeleton and $6$-skeleton, respectively.  By the Hurewicz Theorem, we see also that $\pi_5(\Delta_4;\Z)$ is not generated by maps of spheres into the $5$-skeleton.
\end{remark}

\bibliographystyle{amsalpha}
\bibliography{math}

\end{document}